\documentclass[12pt]{amsart}
\usepackage[foot]{amsaddr}
\usepackage{microtype}
\usepackage{url}
\usepackage{amsthm,amsfonts,amscd,amsmath,amssymb}
\usepackage[margin=1in]{geometry}
\usepackage[english]{babel}
\usepackage{mathtools}
\usepackage{graphicx}
\usepackage{bbm}

\usepackage{tikz}
\usetikzlibrary{arrows,intersections,backgrounds,calc}
\usepackage{pgfplots}	
\pgfplotsset{compat=newest}
\pgfdeclarelayer{background}
\pgfdeclarelayer{foreground}
\pgfsetlayers{background,main,foreground}

\newcommand{\1}{\mathbbm{1}}
\newcommand{\J}{\mathbb{J}}
\newcommand{\expected}{\mathbb{E}}

\DeclareMathOperator{\Tr}{Tr}

\newcommand{\indicator}[1]{1_{ #1 }\,}

\newcommand{\cin}{{c_{\mathrm{in}}}}
\newcommand{\cout}{{c_{\mathrm{out}}}}
\newcommand{\ein}{{m_{\mathrm{in}}}}
\newcommand{\eout}{{m_{\mathrm{out}}}}
\newcommand{\mein}{{\overline{m}_{\mathrm{in}}}}
\newcommand{\meout}{{\overline{m}_{\mathrm{out}}}}
\newcommand{\mm}{\overline{m}}

\renewcommand{\P}{\mathbb{P}}
\newcommand{\Q}{\mathbb{Q}}

\newcommand{\e}{\mathrm{e}}

\newcommand{\ER}{Erd\H{o}s-R\'{e}nyi}
\newcommand{\SBM}{\text{SBM}}

\newcommand{\connectivity}{\gamma}
\newcommand{\overlap}{\alpha}

\newcommand{\dcond}{d_{\mathrm{c}}}
\newcommand{\dcondupper}{d_{\mathrm{c}}^{\mathrm{upper}}}
\newcommand{\dcondlower}{d_{\mathrm{c}}^{\mathrm{lower}}}
\newcommand{\avgdeg}{d}
\newcommand{\Bin}{\mathrm{Bin}}

\newtheorem{theorem}{Theorem}

\newtheorem{corollary}{Corollary}
\newtheorem{lemma}{Lemma}

\title[Thresholds for sparse community detection]{Information-theoretic thresholds for community detection in sparse networks}
\author{Jess Banks \& Cristopher Moore}
\address{Santa Fe Institute, 1399 Hyde Park Road, Santa Fe NM 87501}

\begin{document}
\maketitle

\begin{abstract}
We give upper and lower bounds on the information-theoretic threshold for community detection in the stochastic block model.  Specifically, let $k$ be the number of groups, $d$ be the average degree, the probability of edges between vertices within and between groups be $\cin/n$ and $\cout/n$ respectively, and let $\lambda = (\cin-\cout)/(k\avgdeg)$.  We show that, when $k$ is large, and $\lambda = O(1/k)$, the critical value of $d$ at which community detection becomes possible---in physical terms, the condensation threshold---is 
\[
\dcond = \Theta\!\left( \frac{\log k}{k \lambda^2} \right) \, , 
\]
with tighter results in certain regimes.  Above this threshold, we show that the only partitions of the nodes into $k$ groups are correlated with the ground truth, giving an exponential-time algorithm that performs better than chance---in particular, detection is possible for $k \ge 5$ in the disassortative case $\lambda < 0$ and for $k \ge 11$ in the assortative case $\lambda > 0$.  (Similar upper bounds were obtained independently by Abbe and Sandon.)  Below this threshold, we use recent results of Neeman and Netrapalli (who generalized arguments of Mossel, Neeman, and Sly) to show that no algorithm can label the vertices better than chance, or even distinguish the block model from an \ER\ random graph with high probability.  We also rely on bounds on certain functions of doubly stochastic matrices due to Achlioptas and Naor; indeed, our lower bound on $\dcond$ is their second moment lower bound on the $k$-colorability threshold for random graphs with a certain effective degree.
\end{abstract}

\section{Introduction}
\label{sec:intro}

The Stochastic Block Model (SBM) is a random graph ensemble with planted community structure, where the probability of a connection between each pair of vertices is a function only of the groups or communities to which they belong.  It was originally invented in sociology~\cite{HLL83}; it was reinvented in physics and mathematics under the name ``inhomogeneous random graph''~\cite{Soderberg02,BJR07}, and in computer science as the planted partition problem (e.g.~\cite{mcsherry}).  

Given the current interest in network science, the block model and its variants have become popular parametric models for the detection of community structure.  An interesting set of questions arise when we ask to what extent the communities, i.e., the labels describing the vertices' group memberships, can be recovered from the graph it generates.  In the case where the average degree grows as $\log n$, if the structure is sufficiently strong then the underlying communities can be recovered~\cite{BC09}, and the threshold at which this becomes possible has recently been determined~\cite{abbe-bandeira-hall,abbe-sandon,agarwal-etal}.  Above this threshold, efficient algorithms exist that recover the communities exactly, labeling every vertex correctly with high probability; below this threshold, exact recovery is information-theoretically impossible. 

In the sparse case where the average degree is $O(1)$, finding the communities is more difficult, since we effectively have only a constant amount of information about each vertex.  In this regime, our goal is to label the vertices better than chance, i.e., to find a partition with nonzero correlation or mutual information with the ground truth.  This is sometimes called the \emph{detection} problem to distinguish it from exact recovery.  A set of phase transitions for this problem was conjectured in the statistical physics literature based on tools from spin glass theory.  Some of these conjectures have been made rigorous, while others remain as tantalizing open problems.

\subsection{The Kesten-Stigum bound, information-theoretic detection, and condensation}

Let $k$ be the number of groups, and let the probability of edges between vertices within and between groups be $\cin/n$ and $\cout/n$ respectively.  Assuming each group is of size $n/k$, the average degree is then 
\begin{equation}
\label{eq:avgdeg}
\avgdeg = \frac{\cin + (k-1)\cout}{k} \, . 
\end{equation}
It is convenient to parametrize the strength of the community structure as 
\begin{equation}
\label{eq:lambda}
\lambda = \frac{\cin-\cout}{k\avgdeg} \, .
\end{equation}
As we will see below, this is the second eigenvalue of a transition matrix describing how labels are ``transmitted'' between neighboring vertices.  It lies in the range
\[
-\frac{1}{k-1} \le \lambda \le 1 \, ,
\]
where $\lambda = -1/(k-1)$ corresponds to $\cin = 0$ (also known as the planted graph coloring problem, see below) and $\lambda = 1$ corresponds to $\cout = 0$ where vertices only connect to others in the same group.  We say that block models with $\lambda > 0$ and $\lambda < 0$ are \emph{assortative} and \emph{disassortative} respectively.

The conjecture of~\cite{DKMZ11a,DKMZ11b} is that efficient algorithms exist if and only if we are above the threshold
\begin{equation}
\label{eq:kesten-stigum}
d = \frac{1}{\lambda^2} \, .
\end{equation}
This is known in information theory as the Kesten-Stigum threshold~\cite{Kesten1966,Kesten1966a}, and in physics as the Almeida-Thouless line~\cite{AlmeidaThouless78}.  

Above the Kesten-Stigum threshold, \cite{DKMZ11a,DKMZ11b} claimed that community detection is computationally easy, and moreover that belief propagation---also known in statistical physics as the cavity method---is asymptotically optimal in that it maximizes the fraction of vertices labeled correctly (up to a permutation of the groups).  For $k=2$, this was proved in~\cite{MNS-colt}, and very recently, a type of belief propagation was shown to perform better than chance for all $k$~\cite{abbe-sandon-more-groups}.  In addition, a spectral clustering algorithm based on the non-backtracking operator was conjectured to succeed all the way down to the Kesten-Stigum threshold~\cite{Krzakala13}, and this was proved in~\cite{bordenave-lelarge-massoulie}.

What happens below the Kesten-Stigum threshold is more complicated.  The authors of \cite{DKMZ11a,DKMZ11b} conjectured that for sufficiently small $k$, community detection is information-theoretically impossible when $d < 1/\lambda^2$.  This was proved in the case $k=2$ by~\cite{MNS12}, who established two separate results.  First, they showed that the ensemble of graphs produced by the stochastic block model becomes \emph{contiguous} with that produced by \ER\ graphs of the same average degree, making it impossible even to tell whether or not communities exist with high probability.  Secondly, by relating community detection to the robust reconstruction problem on trees~\cite{janson2004robust}, they showed that for most pairs of vertices the probability, given the graph, that they are in the same group asymptotically approaches $1/2$.  Thus it is impossible, even if we could magically compute the true posterior probability distribution, to label the vertices better than chance. 

On the other hand, \cite{DKMZ11a,DKMZ11b} conjectured that for sufficiently large $k$, namely $k \ge 5$ in the assortative case $\cin > \cout$ and $k \ge 4$ in the disassortative case $\cin < \cout$, there is a ``hard but detectable'' regime where community detection is information-theoretically possible, but computationally hard.  One indication of this is the extreme case where $\cin = 0$: this is equivalent to the planted graph coloring problem where we choose a random coloring of the vertices, and then choose $dn/2$ edges uniformly from all pairs of vertices with different colors.  In this case, we have $\lambda = -1/(k-1)$ and~\eqref{eq:kesten-stigum} becomes $d > (k-1)^2$.  However, while graphs generated by this case of the block model are $k$-colorable by definition, the $k$-colorability threshold for \ER\ graphs grows as $2k \ln k$~\cite{achlioptas-naor}, and falls below the Kesten-Stigum threshold for $k \ge 5$.  In between these two thresholds, we can at least distinguish the two graph ensembles by asking whether a $k$-coloring exists; however, finding one might take exponential time. 

More generally, planted ensembles where some combinatorial structure is built into the graph, and un-planted ensembles such as \ER\ graphs where these structures occur by chance, become distinguishable at a phase transition called \emph{condensation}~\cite{Krzakala2007a}.  Below this point, the two ensembles are contiguous; above it, the Gibbs distribution in the planted model is dominated by a cluster of states surrounding the planted state.  For instance, in random constraint satisfaction problems, the uniform distribution on solutions becomes dominated by those near the planted one; in our setting, the posterior distribution of partitions becomes dominated by those close to the ground truth (although, in the sparse case, with a Hamming distance that is still linear in $n$).  Thus the condensation threshold is also the threshold for information-theoretic community detection.  Below it, even optimal Bayesian inference will do no better than chance, while above it, typical partitions chosen from the posterior will be fairly accurate (though finding these typical partitions might take exponential time).  

We note that some previous results show that community detection is possible below the Kesten-Stigum threshold when the sizes of the groups are unequal~\cite{neeman-netrapalli,zhang-moore-newman}.  In addition, even a vanishing amount of initial information can make community detection possible if the number of groups grows with the size of the network~\cite{kanade-mossel-schramm}.

\subsection{Our contribution}

We give rigorous upper and lower bounds on the information-theoretic threshold for community detection, or equivalently the condensation threshold, bounding the critical average degree as a function of $k$ and $\lambda$.  First, we use a first-moment argument to show that if 
\begin{equation}
\label{eq:d-upper}
d > \dcondupper 
= \frac{2 k \log k}
{(1+(k-1) \lambda) \log (1+(k-1) \lambda) 
+ (k-1)(1-\lambda) \log (1-\lambda)} \, ,
\end{equation}
then, with high probability, the only partitions that are as good as the planted one---that is, which have the expected number of edges within and between groups---have a nonzero correlation with the planted one.  As a result, there is a simple exponential time algorithm for labeling the vertices better than chance: simply test all partitions, and output the first good one.  

We note that $\dcondupper < 1/\lambda^2$ for $k \ge 5$ when $\lambda$ is sufficiently negative, including the case $\lambda = -1/(k-1)$ corresponding to graph coloring discussed above.  Moreover, for $k \ge 11$, there also exist positive values of $\lambda$ for $\dcondupper < 1/\lambda^2$.  Thus for sufficiently large $k$, detectability is information-theoretically possible below the Kesten-Stigum threshold, in both the assortative and disassortative case.  Similar (and somewhat tighter) results were obtained independently by~\cite{abbe-sandon-more-groups}.

We then show that community detection is information-theoretically impossible if
\begin{equation}
\label{eq:d-lower}
d < \dcondlower 
= \frac{2\log(k-1)}{k-1} \frac{1}{\lambda^2}
\, . 
\end{equation}
Here we rely heavily on a recent preprint~\cite{neeman-netrapalli}, who gave a beautiful generalization of the argument of~\cite{MNS12}.  Using the small subgraph conditioning method, they showed that the block model and the \ER\ graph are contiguous whenever the second moment of the ratio between their probabilities---roughly speaking, the number of good partitions in an \ER\ graph---is appropriately bounded.  They also show that this second moment bound implies non-detectability, in that the posterior distribution on any finite collection of vertices is asymptotically uniform.  This reduces the proof of contiguity and non-detectability to a second moment argument, which in turn consists of maximizing a certain function of doubly stochastic matrices.  

Happily, this latter problem was largely solved in~\cite{achlioptas-naor}, who used the second moment method to give nearly-tight lower bounds on the $k$-colorability threshold.  Indeed, our bound~\eqref{eq:d-lower} corresponds to their lower bound on $k$-colorability for $G(n,\avgdeg'/n)$ where $\avgdeg' = \avgdeg \lambda^2 (k-1)^2$.  Intuitively, $\avgdeg'$ is the degree of a random graph in which the correlations between vertices in the $k$-colorability problem are as strong as those in the block model with average degree $\avgdeg$ and eigenvalue $\lambda$.

Our bounds are tight in some regimes, and rather loose in others.  Let $\mu$ denote $(\cin-\cout)/\avgdeg$.  If $\mu$ is constant and $k$ is large, we have
\[
\lim_{k \to \infty} \frac{\dcondupper}{\dcondlower}
= \frac{\mu^2}{(1+\mu) \log (1+\mu) - \mu} \, .
\]
In the limit $\mu = -1$, corresponding to graph coloring, this ratio is $1$, inheriting the tightness of previous upper and lower bounds on $k$-colorability.  For other values of $\mu$, our bounds match up to a multiplicative constant.  In particular, when $k$ is constant and $|\lambda|$ is small, they are about a factor of $2$ apart:
\[
\frac{2\log(k-1)}{k-1} \le \dcond \lambda^2 \le \frac{4 \log k}{k-1} (1+O(k \lambda)) \, . 
\]
When $\lambda \ge 0$ is constant and $k$ is large, we have 
\[
\dcondupper = \frac{2}{\lambda} (1 + O(1/\log k)) \, ,
\]
so that in the limit of large $k$ detectability is possible below the Kesten-Stigum threshold whenever $\lambda < 1/2$.

\section{Our model, notation, and results}

We have $n$ vertices $V=[n]$.  In the usual version of the block model, we start by choosing a partition $\sigma : [n] \to [k]$ uniformly from the $k^n$ possibilities.  We independently include each pair of vertices $(u, v)$ in the edge set $E$ with probability $c_{\sigma(u),\sigma(v)}/n$, where $c$ is a $k \times k$ matrix.  As in much other work on community detection, we focus on the special case 
\begin{equation}
\label{eq:c-cin-cout}
c_{r,s} = \begin{cases} \cin & \mbox{if $r=s$} \\ \cout & \mbox{if $r \ne s$} \, . 
\end{cases}
\end{equation}
We will often find it convenient to assume that the partition $\sigma$ is \emph{balanced}, i.e., that $n$ is divisible by $k$ and that there are $|\sigma^{-1}(r)| = n/k$ vertices in each group $r$.  Of course, this is true within an $o(n)$ error term with high probability. 

%

Recalling that the expected average degree is 
\[
\avgdeg = \frac{\cin + (k-1)\cout}{k} \, ,
\]
we will find it useful to define a doubly stochastic matrix, 
\begin{align}
	\connectivity = \frac{c}{k\avgdeg} 
	= \frac{1}{k\avgdeg} \begin{pmatrix}
			\cin & {} & \cout \\
			{} & \ddots & {} \\
			\cout & {} & \cin
	\end{pmatrix} \, .
	\label{eq:connectivity}
\end{align}
We can think of $\connectivity$ as a transition matrix in a Markov random field.  All else equal, if $(u,v) \in E$ and $\sigma(u) = r$, then $\connectivity_{rs}$ is the probability that $\sigma(v) = s$.  Notice that 
\begin{equation}
\label{eq:id}
\connectivity = \lambda \1 + (1-\lambda) \frac{\J}{k} \, , 
\end{equation}
where $\J$ is the matrix of all $1$s, and where 
\[
\lambda = \frac{\cin-\cout}{k\avgdeg} 
\]
is $\connectivity$'s second eigenvalue.  We can think of $\lambda$ as the probability that information is transmitted from $u$ to $v$: with probability $\lambda$ we copy $u$'s group label to $v$, and with probability $1-\lambda$ we choose $v$'s group uniformly.  The parameter $\lambda$ interpolates between the case $\lambda=1$ where all edges are within-group, to an \ER\ graph where $\lambda = 0$ and edges are placed uniformly at random, to $\lambda < 0$ where edges are more likely between groups than within them.  This gives a useful reparametrization of the model in terms of $c$ and $\lambda$, where
\begin{align}
	\cin &= \avgdeg (1 + (k-1) \lambda) \nonumber \\
	\cout &= \avgdeg (1 - \lambda) \, .
	\label{eq:reparam}
\end{align}


The community detection problem is to recover the planted partition $\sigma$ from the graph $G=(V,E)$.  We define the degree to which an algorithm succeeds as follows.  Given another partition $\tau$, we define the overlap matrix 
\[
\overlap_{rs}
= \frac{k}{n} | \sigma^{-1}(r) \cap \tau^{-1}(s) | \, . 
\]
Since $\sigma$ is balanced, this is the fraction of group $r$ (according to $\sigma$) that is in group $s$ (according to $\tau)$.  If $\tau$ is balanced as well, then $\overlap$ is doubly stochastic.  The overlap $\beta$ is then the fraction of vertices labeled correctly, maximized over all permutations $\pi$ of the groups,
\[
\beta = \frac{1}{k} \max_\pi \sum_r \overlap_{r,\pi(r)} = \frac{1}{k} \max_\pi \Tr \pi \overlap \, , 
\]
where in the second expression we interpret $\pi$ as a permutation matrix.  We define the information-theoretic threshold $\dcond$ as the value of $\avgdeg$ above which a (possibly exponential-time) algorithm exists that performs better than chance, i.e., which achieves an overlap bounded above $1/k$.  

As in~\cite{MNS12,neeman-netrapalli} we will compare the stochastic block model to the \ER\ random graph, and ask to what extent these two probability distributions differ.  We use $\P$ for the probability of a graph in the block model, and $\Q$ for the \ER\ model $G(n,\avgdeg/n)$.  Thus
\begin{align*}
\P(G \mid \sigma) 
&= \prod_{(u,v) \in E} \frac{c_{\sigma(u),\sigma(v)}}{n} \prod_{(u,v) \notin E} \left( 1-\frac{c_{\sigma(u),\sigma(v)}}{n} \right) \\
\P(G) &= \frac{1}{k^n} \sum_{\sigma \in [k]^n} \P(G \mid \sigma) \\
\Q(G) &= \prod_{(u,v) \in E} \frac{\avgdeg}{n} \prod_{(u,v) \notin E} \left( 1-\frac{\avgdeg}{n} \right) \, . 
\end{align*}

We now state our main result.

\begin{theorem}  
The information-theoretic transition for community detection in the block model with parameters $\cin, \cout$ is bounded above and below by 
\[
\dcondlower \le \dcond \le \dcondupper
\]
where
\begin{align}
\dcondupper 
&= \frac{2 k \log k}
{(1+(k-1) \lambda) \log (1+(k-1) \lambda) 
+ (k-1)(1-\lambda) \log (1-\lambda)} \\
\dcondlower 
&= \frac{2\log(k-1)}{k-1}\frac{1}{\lambda^2} \, ,
\end{align}
where $\avgdeg$ and $\lambda$ are defined as in~\eqref{eq:avgdeg} and~\eqref{eq:lambda}.  
That is, if $\avgdeg > \dcondupper$ there is an exponential-time algorithm that achieves overlap bounded above $1/k$; while if $\avgdeg < \dcondlower$ the block model is contiguous to the \ER\ random graph $G(n,\avgdeg/n)$, and no algorithm can achieve overlap bounded above $1/k$.
\end{theorem}

\begin{corollary}
For $k \ge 5$, community detection is information-theoretically possible below the Kesten-Stigum threshold $\avgdeg = 1/\lambda^2$ for $\lambda$ sufficiently negative.  For $k \ge 11$, there exist positive values of $\lambda$ for which this holds as well.
\end{corollary}

The bounds $\dcondupper$ and $\dcondlower$ are obtained in~\S\ref{sec:first} and~\S\ref{sec:second} respectively.

\section{The Upper Bound}
\label{sec:first}

Our upper bound on the detectability threshold hinges on the following observation.  With high probability, a graph generated by the SBM has at least one balanced partition, close to the the planted one, where the number of within-group and between-group edges $\ein$ and $\eout$ are close to their expectations.  That is, 
\begin{equation}
\label{eq:in-close}
|\ein - \mein| < n^{2/3} 
\quad \text{and} \quad 
|\eout - \meout| < n^{2/3}
\end{equation}
where
\begin{align} 
\mein &= \frac{\cin}{2k} \,n
= \frac{\avgdeg (1+(k-1)\lambda)}{2k} \,n \nonumber \\
\meout &= \frac{(k-1) \cout}{2k} \,n
= \frac{\avgdeg (k-1)(1-\lambda)}{2k} \,n \, .
\label{eq:mein-def}
\end{align}
This follows from standard concentration inequalities on the binomial distribution: the number of vertices in each group in $\sigma$ is w.h.p.\ $n/k + o(n^{2/3} / \log n)$, in which case~\eqref{eq:in-close} holds w.h.p.  Since the maximum degree is w.h.p.\ less than $\log n$, we can modify $\sigma$ to make it balanced while changing $\ein$ and $\eout$  by $o(n^{2/3})$.

Call such a partition \emph{good}.  We will show that if $\avgdeg > \dcondupper$ all good partitions are correlated with the planted one.  As a result, there is an exponential algorithm that performs better than chance: simply use exhaustive search to find a good partition, and output it.

\subsection{Distinguishability from $G(n,\avgdeg/n)$}

As a warm-up, we show that if $\avgdeg > \dcondupper$ the probability that an \ER\ graph has a good partition is exponentially small, so the two distributions $\P$ and $\Q$ are asymptotically orthogonal.

Let $G$ be a graph generated by $G(n,\avgdeg/n)$.  We condition on the high-probability event that it has $m$ edges with $|m - \mm| < n^{2/3}$ with
\[
\mm = \mein + \meout = \avgdeg n/2 \, ,
\]
in which case $G$ is chosen from $G(n,m)$.  Since $G$ is sparse, we can think of its $m$ edges as chosen uniformly with replacement from the $n^2$ possible ordered pairs.  With probability $\Theta(1)$ the resulting graph is simple, with no self-loops or multiple edges, and hence uniform in $G(n,m)$.  Thus any event that holds with high probability in the resulting model holds with high probability in $G(n,m)$ as well.  Call this model $G'(n,m)$.  

For a given balanced partition $\sigma$, the probability in $G'(n,m)$ that a given edge has its endpoints in the same group is $1/k$.  Thus, up to subexponential terms resulting from summing over the $n^{2/3}$ possible values of the error terms, the probability that a given $\sigma$ is good is 
\[
\Pr[\Bin(\mm,1/k) = \mein] = {\mm \choose \mein} (1/k)^\mein (1-1/k)^{\meout} \, . 
\]
The rate of this large-deviation event is given by the Kullback-Leibler divergence between binomial distributions with success probability $1/k$ and $\mein/\mm$, 
\begin{align*}
\lim_{\mm \to \infty} \frac{1}{\mm} \log \Pr[\Bin(\mm,1/k) = \mein]
&= -\frac{\mein}{\mm} \log \frac{\mein/\mm}{1/k} - \frac{\meout}{\mm} \log \frac{\meout/\mm}{1-1/k} \\
&= -\frac{\cin}{k \avgdeg} \log \frac{\cin}{\avgdeg} - \left( 1 - \frac{\cin}{k \avgdeg} \right) \log \frac{k \avgdeg-\cin}{\avgdeg (k-1)} \, , 
\end{align*}
where we used $\mein/\mm = \cin/(k \avgdeg)$ and $\meout/\mm = 1-\cin/(k \avgdeg)$.  Writing this in terms of $\avgdeg$ and $\lambda$ as in~\eqref{eq:reparam} and simplifying gives
\begin{equation}
\label{eq:good-rate}
\lim_{n \to \infty} \frac{1}{n} \Pr[\mbox{$\sigma$ is good}] 
= - \frac{\avgdeg}{2k} \big[
(1+(k-1) \lambda) \log (1+(k-1) \lambda) 
+ (k-1)(1-\lambda) \log (1-\lambda)
\big] 
\, . 
\end{equation}
Now, by the union bound, since there are at most $k^n$ balanced partitions, the probability that any good partitions exist is exponentially small whenever the function in~\eqref{eq:good-rate} is less than $-\log k$.  This tells us that the block model is distinguishable from an \ER\ graph whenever 
\[
\avgdeg > \dcondupper 
= \frac{2 k \log k}
{(1+(k-1) \lambda) \log (1+(k-1) \lambda) 
+ (k-1)(1-\lambda) \log (1-\lambda)} \, ,
\]
As noted above, the limit $\lambda = -1/(k-1)$ corresponds to the planted graph coloring problem.  In this case $\dcondupper$ is simply the first-moment upper bound on the $k$-colorability threshold, 
\[
\dcondupper = \frac{2 \log k}{-\log (1-1/k)} < 2 k \log k \, .
\]

\subsection{All good partitions are accurate} 

Next we show that, if $d > \dcondupper$, with high probability any good partition is correlated with the planted one.  Essentially, the previous calculation for $G(n,m)$ corresponds to counting good partitions which are uncorrelated with $\sigma$, i.e., which have a flat overlap matrix $\overlap = \J/k$.  We will show that in order for a good partition to exist, its overlap matrix must have some large entries, in which case its overlap with $\sigma$ is bounded above $1/k$.  

Given a balanced partition $\tau$, let $\ein$ and $\eout$ denote the number of edges $(u,v)$ with $\tau(u) = \tau(v)$ and $\tau(u) \ne \tau(v)$ respectively.  As in the previous section, we say that $\tau$ is \emph{good} if~\eqref{eq:in-close} holds, i.e., $|\ein-\mein|, |\eout-\meout| < n^{2/3}$ where $\mein$ and $\meout$ are given by~\eqref{eq:mein-def}.

\begin{theorem}
\label{thm:good-accurate}
Let $G$ be generated by the stochastic block model with parameters $\cin$ and $\cout$, and let $\avgdeg$ and $\lambda$ be defined as in~\eqref{eq:avgdeg} and~\eqref{eq:lambda}.  If $\avgdeg > \dcondupper$ then, with high probability, any good partition has overlap at least $\beta > 1/k$ with the planted partition $\sigma$, where $\beta$ is the smallest root of 
\begin{equation}
\label{eq:good-accurate}
\avgdeg = \frac{2k \bigl( h(\beta) + (1-\beta) \log (k-1) \bigr)}
{
(1+(k-1)\lambda) \log \frac{1+(k-1) \lambda}{1 + (k\beta-1) \lambda}
+ (k-1)(1-\lambda) \log \frac{(k-1)(1-\lambda)}{k - 1 - (k\beta-1) \lambda}
}
\end{equation}
where $h = -\beta \log \beta - (1-\beta) \log (1-\beta)$ is the entropy function.  Therefore, an exponential-time algorithm exists that w.h.p.\ achieves overlap at least $\beta$.
\end{theorem}

We note that the right-hand side of~\eqref{eq:good-accurate} is an increasing function of $\beta$, and that it coincides with $\dcondupper$ when $\beta=1/k$.

\begin{proof}
We start by conditioning on the high-probability event that $G$ has $m$ edges, where $|m-\mm| < n^{2/3}$ and $\mm = \avgdeg n/2$.  Call the resulting model $G_\SBM(n,m)$ (with the matrix of parameters $c$ implicit).  It consists of the distribution over all simple graphs with $m$ edges, with probability proportional to $\P(G \mid \sigma)$.

In analogy with the model $G'(n,m)$ defined above, we consider another version of the block model where the $m$ edges are chosen independently as follows.  
For each edge, we first choose an ordered pair of groups $r, s$ with probability proportional to $c_{rs}$, i.e., with probability $\connectivity_{rs}/k$ where $\connectivity = c/(k \avgdeg)$ is the doubly stochastic matrix defined in~\eqref{eq:connectivity}.  We then choose the endpoints $u$ and $v$ uniformly from $\sigma^{-1}(r)$ and $\sigma^{-1}(s)$ (with replacement if $r=s$).  
Call this model $G'_\SBM(n,m)$.  In the sparse case $\avgdeg = O(1/n)$, the resulting graph is simple with probability $\Theta(1)$, in which event it is generated by $G_\SBM(n,m)$.  Thus any event that holds with high probability in $G'_\SBM(n,m)$ holds with high probability in $G_\SBM(n,m)$ as well.

Now fix a balanced partition $\tau$, and let $q$ denote the probability that an edge $(u,v)$ chosen in this way is within-group with respect to $\tau$.  Recall that the overlap matrix $\overlap_{st}$ is the probability that $\tau(u)=t$ if $u$ is chosen uniformly from those with $\sigma(u)=s$.  Up to $O(1/n)$ terms, the events that $\tau(u)=t$ and $\tau(v)=t$ are independent.  Thus in the limit $n \to \infty$,
\begin{align*}
q = \Pr[\tau(u)=\tau(v)] 
&= \sum_{r,s,t} \Pr[\sigma(u)=r \wedge \sigma(v)=s \wedge \tau(u)=\tau(v)=t] \\
&= \frac{1}{k} \sum_{r,s,t} \connectivity_{rs} \overlap_{rt} \overlap_{st} \\
&= \frac{1}{k} \Tr \overlap^\dagger \connectivity \overlap \, ,
\end{align*}
where $\dagger$ denotes the matrix transpose.  Using~\eqref{eq:id} and $\J \overlap = \overlap \J = \J$, this gives
\[
q = \frac{1 + (|\overlap|^2-1) \lambda}{k} \, .
\]
where $|\overlap|$ denotes the Frobenius norm, 
\[
|\overlap|^2 = \Tr \overlap^\dagger \overlap = \sum_{r,s} \overlap_{rs}^2 \, .
\]
When $\tau$ and $\sigma$ are uncorrelated and $\overlap = \J/k$, we have $q=1/k$ as in the previous section.  When $\sigma = \tau$ and $\overlap = \1$, we have $q = \cin/(k\avgdeg) = (1+(k-1)\lambda)/k$.

For $\tau$ to be good, we need $|\ein - \mein| < n^{2/3}$.  Since $|m-\mm| < n^{2/3}$ as well, up to subexponential terms the probability that $\tau$ is good is 
\[
\Pr[\Bin(\mm,q) = \mein] = {\mm \choose \mein} q^\mein (1-q)^{\meout} \, . 
\]
The rate at which this occurs is again a Kullback-Leibler divergence, between binomial distributions with success probabilities $q$ and $\mein/\mm = \cin/(k \avgdeg)$.  Following our previous calculations gives 
\begin{align}
&\lim_{n \to \infty} \frac{1}{n} \log \Pr[\Bin(\mm,q) = \mein] 
\label{eq:tau-good} \\
&= -\frac{\avgdeg}{2} \left(
\frac{\cin}{k \avgdeg} \log \frac{\cin}{qk\avgdeg} + \left( 1 - \frac{\cin}{k \avgdeg} \right) \log \frac{1-\cin/k\avgdeg}{1-q} 
\right) 
\nonumber \\
&= -\frac{\avgdeg}{2k} \left[
(1+(k-1)\lambda) \log \frac{1+(k-1) \lambda}{qk}
+ (k-1)(1-\lambda) \log \frac{(k-1)(1-\lambda)}{k(1-q)} 
\right] 
\nonumber \\
&= -\frac{\avgdeg}{2k} \left[
(1+(k-1)\lambda) \log \frac{1+(k-1) \lambda}{1 + (|\overlap|^2-1) \lambda}
+ (k-1)(1-\lambda) \log \frac{(k-1)(1-\lambda)}{k - 1 - (|\overlap|^2-1) \lambda}
\right] 
\, . \nonumber 
\end{align}

We pause to prove a lemma which relates the Frobenius norm to the overlap.  This bound is far from tight except in the extreme cases $\overlap = \J/k$ and $\overlap = \1$, but it lets us derive an explicit lower bound on the overlap of a good partition.  Recall that the overlap $\beta$ is the maximum, over all permutations $\pi$, of $(1/k) \sum_r \overlap_{r,\pi(r)} = (1/k) \Tr \pi \overlap$.  
\begin{lemma}
\label{lem:overlap} 
$|\overlap|^2 \le k\beta$.
\end{lemma}

\begin{proof}
Since $\overlap$ is doubly stochastic, Birkhoff's theorem tells us it can be expressed as a convex combination of permutation matrices, 
\[
\overlap = \sum_\pi a_\pi \pi 
\quad \text{where} \quad
\sum_\pi a_\pi = 1 \, . 
\]
Thus 
\begin{align}
	|\overlap|^2 
	= \Tr \overlap^{\dagger} \overlap 
	= \Tr \left( \sum_\pi a_\pi \pi^{-1} \right) \overlap 
	= \sum_\pi a_\pi \Tr \pi^{-1} \overlap 
	\le \max_\pi \Tr \pi^{-1} \overlap 
	= k\beta \, .
\end{align}
\end{proof}

The function in~\eqref{eq:tau-good} is an increasing function of $\lambda$, since as $\lambda$ increases the distributions $\Bin[\mm,q]$ and $\Bin[\mm,\cin/(k\avgdeg)]$ become closer in Kullback-Leibler distance.  Thus if $\tau$ has overlap $\beta$,  Lemma~\ref{lem:overlap} implies 
\begin{align}
\label{eq:tau-good-2}
\lim_{n \to \infty} &\frac{1}{n} \Pr[\mbox{$\tau$ is good}] 
\nonumber \\
&\le -\frac{\avgdeg}{2k} \left[
(1+(k-1)\lambda) \log \frac{1+(k-1) \lambda}{1 + (k\beta-1) \lambda}
+ (k-1)(1-\lambda) \log \frac{(k-1)(1-\lambda)}{k - 1 - (k\beta-1) \lambda}
\right] \, . 
\end{align}

For fixed $\sigma$, the number of balanced partitions $\tau$ with overlap matrix $\overlap$ is the number of ways to partition each group $\sigma^{-1}(r)$ so that there are $\overlap_{rs} n/k$ vertices in $\sigma^{-1}(r) \cap \tau^{-1}(s)$:
\[
\prod_{r=1}^k {n/k \choose \{ \overlap_{rs} n / k \mid 1 \le s \le k \} } 
= \prod_{r=1}^k \frac{(n/k)!}{\prod_s (\overlap_{r,s} n/k)!} \le \e^{n H(\overlap)} \, , 
\]
where $H(\overlap)$ is the average entropy of the rows of $\overlap_{rs}/k$, 
\[
H(\overlap) = - \frac{1}{k} \sum_{r,s} \overlap_{rs} \log \overlap_{rs} \, . 
\]
By the union bound, the probability that there are any good partitions with overlap matrix $\overlap$ is exponentially small whenever the sum of $H(\overlap)$ and the right-hand side of~\eqref{eq:tau-good-2} is negative.  For a fixed overlap $\beta$, maximized by the permutation $\pi$, the entropy $H(\overlap)$ is maximized when 
\[
\overlap_{rs} = \begin{cases}
\beta & \mbox{if $s=\pi(r)$} \\
(1-\beta)/(k-1) & \mbox{if $s \ne \pi(r)$} \, , 
\end{cases}
\]
so we have 
\begin{equation}
\label{eq:entropy-bound}
H(\overlap) \le h(\beta) + (1-\beta) \log (k-1) \, . 
\end{equation}
Combining the bounds~\eqref{eq:tau-good-2} and~\eqref{eq:entropy-bound}, and requiring that their sum is at least zero, completes the proof.
\end{proof}


\subsection{Detection below the Kesten-Stigum bound}

In \S\ref{sec:intro} we commented on the asymptotic behavior of $\dcondupper$ in various regimes.  In Table~\ref{tab:1} we give, for various values of $k$, the point $\lambda^*$ at which $\dcondupper = 1/\lambda^2$; then $\dcondupper < 1/\lambda^2$ for $\lambda < \lambda^*$.  As stated above, in the limit $k \to \infty$ we have $\dcondupper = 2/\lambda$, so $\lambda^*$ tends to $1/2$.

\begin{table}
\begin{center}
$
\begin{array}{|c|c|c|c|c|c|c|c|c|c|c|c|}
\hline
k & 5 & 6 & 7 & 8 & 9 & 10 & 11 & 20 & 100 & 1000 & 10^4 \\
\lambda^* 
& -0.239
& -0.166
& -0.112
& -0.070
& -0.036
& -0.08 
& 0.014 
& 0.127
& 0.286
& 0.372
& 0.410
\\
\hline
\end{array}
$
\end{center}
\bigskip
\caption{For $\lambda < \lambda^*$ we have $\dcondupper < 1/\lambda^2$, so that community detection is information-theoretically possible below the Kesten-Stigum bound.  For $k \ge 5$, this holds in the sufficiently disassortative case, including planted graph coloring where $\lambda = -1/(k-1)$.  For $k \ge 11$, it occurs throughout the disassortative range $\lambda < 0$, and in some assortative cases.}
\label{tab:1}
\end{table}


\section{The Lower Bound}
\label{sec:second}

In this section we derive a lower bound $\dcondlower$ on the information-theoretic transition, using a sufficient condition established by Neeman and Netrapalli~\cite{neeman-netrapalli}.  They generalize the small-subgraph conditioning argument of Mossel, Neeman, and Sly~\cite{MNS12} to show that both contiguity and non-reconstruction follow whenever a second moment bound holds for the ratio $\P/\Q$ between the block model and the \ER\ graph.  Specifically, they show that this bound implies two separate results.  First, the \ER\ and block model distributions are mutually contiguous---any event that holds with high probability in one model holds with high probability in the other, so there is no statistical test capable of distinguishing between these models from a single sample.  Second, the planted partition in the block model becomes non-reconstructible---no algorithm for determining the ground truth group labels performs better than chance. 

Using the Laplace method, this second moment bound holds if a certain $\Phi$ of doubly stochastic matrices (parametrized by $k,\avgdeg$, and $\lambda$) is maximized by its value at the flat matrix $\J/k$.  For completeness, and since our notation is rather different from theirs, we derive this function here for the case~\eqref{eq:c-cin-cout} of the block model.  The function $\Phi$ is a combination of an entropy term $H(\overlap)$, which is maximized at $\J/k$, and a correlation or ``energy'' term which is maximized at $\overlap = \1$.  This kind of maximization problem was studied extensively by Achlioptas and Naor~\cite{achlioptas-naor} on the way to proving their lower bound on the $k$-colorability threshold, allowing us to relate this problem to theirs.

\subsection{The second moment, contiguity, and non-reconstructibility}

We apply the following theorem of Neeman and Netrapalli, translated into our notation.
\begin{theorem}
\cite[Theorems 3.7 and 3.9]{neeman-netrapalli} Let
	\begin{align}
		\Phi(\overlap) = H(\overlap) - \log k + \frac{\avgdeg \lambda^2}{2}\left(|\overlap|^2 - 1\right) \, ,
		\label{eq:phi}
	\end{align}
where $H(\overlap) = -(1/k) \sum_{r,s} \overlap_{rs} \log \overlap_{rs}$ 
and $|\overlap|$ denotes the Frobenius norm. If $\Phi(\overlap) \le 0$ for all doubly stochastic $\overlap$, then (i) $\P$ and $\Q$ are mutually contiguous, and (ii) $\P$ is non-reconstructible.
\end{theorem}

For completeness, we show how the function $\Phi$ arises in the second moment of the ratio $\P/\Q$.  For ease of exposition, we will ignore subexponential terms; however, as in~\cite{achlioptas-moore,achlioptas-naor} these cancel out when we take the polynomial factors of the Laplace approximation into account.  
Recall from~\S1 that the probability in the block model of $G$ given a particular planted partition $\sigma$ is 
\begin{align*}
	\P(G\mid\sigma) = \prod_{(u,v) \in E} \left(\frac{c_{\sigma(u),\sigma(v)}}{n}\right) \prod_{(u,v)\notin E}\left(1 - \frac{c_{\sigma(u),\sigma(v)}}{n}\right),
\end{align*}
and the total probability is
\begin{align*}
	\P(G) = \frac{1}{k^n}\sum_{\sigma} \P(G\mid\sigma) \, .
\end{align*}
In contrast, the probability of $G$ in the \ER\ model $G(n,\avgdeg/n)$ is
\begin{align*}
	\Q(G) = \prod_{(u,v)\in E} \frac{\avgdeg}{n} \prod_{(u,v)\notin E}\left(1 - \frac{\avgdeg}{n}\right).
\end{align*} 
Combining these, the square of the ratio $\P/\Q$ for fixed $G$ is
\begin{align*}
	\left( \frac{\P(G)}{\Q(G)} \right)^2 
	%
	&= \frac{1}{k^{2n}} \sum_{\sigma,\tau} \prod_{(u,v)} \left( \indicator{(u,v)\in E}\, \frac{c_{\sigma(u),\sigma(v)} c_{\tau(u),\tau(v)}}{\avgdeg^2} + \indicator{(u,v)\notin E}\,\left(\frac{\left(1 - \frac{c_{\sigma(u),\sigma(v)}}{n}\right)\left(1 - \frac{c_{\tau(u),\tau(v)}}{n}\right)}{\left(1 - \frac{\avgdeg}{n}\right)^2}\right)\right) .
\end{align*}

In the \ER\ distribution $\Q$, edges are independent events with $\expected_{\Q}\indicator{(u,v)\in E} = \avgdeg/n$.  Thus if $G$ is drawn from $\Q$, the second moment of $\P/\Q$ is
\begin{align*}
	\expected_{\Q} \left(\frac{\P(G)}{\Q(G)}\right)^2 
	%
	%
	&= \frac{1}{k^{2n}} \sum_{\sigma,\tau} \prod_{(u,v)} \left(\frac{c_{\sigma(u),\sigma(v)}c_{\tau(u),\tau(v)}}{dn} 
	+ \frac{\left(1 - \frac{c_{\sigma(u),\sigma(v)}}{n}\right)\left(1 - \frac{c_{\tau(u),\tau(v)}}{n}\right)}{1 - \frac{\avgdeg}{n}} 
	\right) \\ 
	&= \frac{1}{k^{2n}} \sum_{\sigma,\tau} \prod_{(u,v)} \left(
	1 + \frac{\avgdeg}{n} 
	\left( \frac{c_{\sigma(u),\sigma(v)}}{\avgdeg} - 1 \right) 
	\left( \frac{c_{\tau(u),\tau(v)}}{\avgdeg} - 1 \right) 
	+ O(1/n^2) 
	\right) \\
	&= \frac{1}{k^{2n}} \sum_{\sigma,\tau} \exp\!\left[ \sum_{(u,v)} \left(
	 \frac{\avgdeg}{n} 
	 \left( \frac{c_{\sigma(u),\sigma(v)}}{\avgdeg} - 1 \right) 
	 \left( \frac{c_{\tau(u),\tau(v)}}{\avgdeg} - 1 \right) 
	+ O(1/n^2) \right) \right] 
	\, . 
	%
\end{align*}
We can rewrite this expression in a helpful way. For each pair of partitions $\sigma, \tau$, let us replace the product over vertices with a product over groups; let us assume that $\sigma$ and $\tau$ are both nearly balanced, so that $\overlap$ is doubly-stochastic.  Since there are $\overlap_{rs} n/k$ vertices in $\sigma^{-1}(r) \cap \tau^{-1}(s)$, we have 
\begin{align}
\label{eq:second-sum}
	\expected_{\Q} \left(\frac{\P(G)}{\Q(G)}\right)^2
	&\approx \frac{1}{k^{2n}} \sum_{\sigma,\tau} \exp\!\left[ 
	\,\frac{\avgdeg n}{2} 
	\sum_{r,s,r',s'} 
	\frac{\overlap_{rr'} \overlap_{ss'}}{k^2}
	 \left( \frac{c_{r,s}}{\avgdeg} - 1 \right) 
	 \left( \frac{c_{r's'}}{\avgdeg} - 1 \right) 
	\right] 
	\, . 
\end{align}
where the factor of $1/2$ avoids double-counting of the unordered pairs $(u,v)$.  

Now recall that $\connectivity = c/(k \avgdeg)$ is doubly stochastic, and that $\connectivity = \lambda \1 + (1-\lambda) \J/k$.  Using $\J \overlap = \overlap \J = \J$, we can write the function in the exponential as 
\begin{align}
\sum_{r,s,r',s'} \overlap_{rr'} \overlap_{ss'} \left( \connectivity_{rs} - 1/k \right) \left( \connectivity_{r's'} - 1/k \right) 
&= \Tr \left[ \overlap^\dagger \left( \connectivity - \J/k \right) \overlap \left( \connectivity - \J/k \right) \right] \nonumber \\
&= \lambda^2 \Tr \left[ \overlap^\dagger \left( \1 - \J/k \right) \overlap \left( \1 - \J/k \right) \right] \nonumber \\
&= \lambda^2 \Tr \left[ \overlap^\dagger \overlap - \J/k \right] \nonumber \\
&= \lambda^2 \left( |\overlap|^2 - 1 \right) \, . 
\label{eq:cute}
\end{align}

Finally, as in other applications of the second moment method, we can approximate the sum over partitions as an integral over all doubly-stochastic overlap matrices $\overlap$, weighted by the number of partition pairs which realize each overlap.  This is the number $k^n$ of partitions $\sigma$, times the number of partitions $\tau$ with overlap matrix $\overlap$.  Following~\S3, this gives a weight $k^n \e^{nH(\overlap)} = \e^{n(H(\overlap)+\log k)}$.  Combining this with~\eqref{eq:second-sum} and~\eqref{eq:cute}, we have (where $\sim$ hides polynomial terms) 
\begin{align}
	\expected_{\Q} \left( \frac{\P}{\Q} \right)^2 
&\sim \frac{1}{k^{2n}} \int d\overlap\, \exp\!\left[ 
n \left( H(\overlap) + \log k + \frac{\avgdeg \lambda^2}{2} \left( |\overlap|^2 - 1 \right) \right) 
\right] \\
&= \int d\overlap\, \exp\!\left[ 
n \left( H(\overlap) - \log k + \frac{\avgdeg \lambda^2}{2} \left( |\overlap|^2 - 1 \right) \right) 
\right] \\
&= \int d\overlap\, \exp\!\left[ n\Phi(\overlap) \right] \, .
\label{eq: second moment scaling}
\end{align}

The integral in~\eqref{eq: second moment scaling} is dominated in the limit $n \to \infty$ by the maximum of $\Phi$, assuming this is a quadratic maximum (i.e., with a Hessian which is negative definite).  In particular, at the flat matrix $\overlap = \J/k$, where $\sigma$ and $\tau$ are roughly independent, we have $H(\overlap) = \log k$, $|\overlap|^2 = 1$, and $\Phi(\overlap) = 0$.  Thus if $\Phi(\overlap) \le 0$ for all $\overlap$, the second moment $\expected_\Q (\P/\Q)^2$ is bounded by a constant.  Using further reasoning, including the small subgraph conditioning method, Neeman and Netrapalli~\cite{neeman-netrapalli} show that this then implies contiguity between the block model and $G(n,\avgdeg/n)$, and moreover that the block model is non-reconstructible.

\subsection{Maximizing $\Phi$}

Achlioptas and Naor, in the process of proving a lower bound on the $k$-coloring threshold for \ER\ graphs, develop substantial machinery for optimizing $\Phi$-like functions over the Birkhoff polytope of doubly stochastic matrices~\cite{achlioptas-naor}. 
Specifically, they relax the problem to maximizing over all row-stochastic matrices, and show that the maximizer is then a mixture of uniform rows and rows where all but one of the entries are identical.  Although their bound is quite general, we quote here their results for the entropy.  (Note that their definition of $H(\overlap)$ and ours differ by a factor of $k$.)

\begin{theorem}
\cite[Theorem 9]{achlioptas-naor} Let $\overlap$ be doubly stochastic with $|\overlap|^2 = \rho$.  Then
\begin{equation}
H(\overlap) \le \max_{ m\in\left[0,\frac{k(k-\rho)}{k-1}\right] } 
\left\{ 
\frac{m}{k} \log k 
+ \left( 1-\frac{m}{k} \right) f\!\left(\frac{k\rho-m}{k(k-m)}\right)\right\} \, ,
\end{equation}
where
\begin{align*}
	f(r) = h\!\left(\frac{1 + \sqrt{(k-1)(kr - 1)}}{k}\right) + (k-1) \,h\!\left(\frac{1 - \frac{1 + \sqrt{(k-1)(kr - 1)}}{k}}{k-1}\right) 
\end{align*}
and $h(x) = -x \log x$. 
\end{theorem}

With this result in hand and using $f(1/k) = k h(1/k) = \log k$, we know that for all $\overlap$ with $|\overlap|^2 = \rho$, 
\begin{align*}
\Phi(\overlap) 
\le 
\left(1 - \frac{m}{k} \right)  \left( f\!\left(\frac{k\rho-m}{k(k-m)}\right) - f(1/k) \right)
+ \frac{\avgdeg \lambda^2}{2}(\rho - 1)
\end{align*}
for $m \in [0,k(k-\rho)/(k-1)]$.  Achlioptas and Naor determined the value of $\avgdeg \lambda^2/2$ for which the right-hand side is less than or equal to zero for all $m$ in this interval and all $\rho \in [1,k]$.  
\begin{lemma}
\cite[Proof of Theorem 7]{achlioptas-naor} 
When $\delta < (k-1) \log (k-1)$,
\begin{align*}
	\frac{\delta(\rho - 1)}{(k-1)^2} \le \left(1 - \frac{m}{k}\right)\left( f(1/k) - f\!\left(\frac{k\rho - m}{k(k-m)}\right)\right)
\end{align*}
for all $m \in [0,k(k-\rho)/(k-1)]$ and all $\rho \in [1,k]$.  
\end{lemma}
Our lower bound is an immediate corollary of this lemma. Substituting $\delta = d \lambda^2 (k-1)^2 / 2$ and solving for $d$ gives
\begin{align}
	\dcondlower = \frac{2\log(k-1)}{k-1}\frac{1}{\lambda^2} \, .
\end{align}

As we commented in~\S\ref{sec:intro}, this corresponds to the lower bound on the $k$-colorability threshold of $G(n,\avgdeg'/n)$ where $\avgdeg' = 2\delta = \avgdeg \lambda^2 (k-1)^2$, scaling the eigenvalue on each edge to $\lambda$ from its value $-1/(k-1)$ for $k$-cooring.  This fits with the Kesten-Stigum threshold as well, since the amount of information (appropriately defined) transmitted along each edge is proportional to $\lambda^2$~\cite{janson2004robust}.

\section{Conclusions}
\label{sec:conclusion}

We (and, independently, \cite{abbe-sandon-more-groups}) have shown that community detection is information-theoretically possible below the Kesten-Stigum threshold.  However, we have not given any evidence that it is computationally hard.  Of course, we cannot hope to prove this without knowing that $\mathrm{P} \ne \mathrm{NP}$, but we could hope to prove that certain classes of algorithms take exponential time.  In particular, we could show that Monte Carlo algorithms or belief propagation take exponential time to find a good partition, assuming their initial states or messages are uniformly random.  

Physically, we believe this occurs because there is a free energy barrier between a ``paramagnetic'' phase of partitions which are essentially random, and a ``ferromagnetic'' or ``retrieval'' phase which is correlated with the planted partition~\cite{DKMZ11a,DKMZ11b,zhang-moore}.  Proving this seems within reach: rigorous results have been obtained in random constraint satisfaction problems~\cite{achlioptas-coja-oghlan,coja-oghlan-efthymiou} showing that solutions become clustered with $O(n)$ Hamming distance and $O(n)$ energy barriers between them, and that Markov chain Monte Carlo algorithms take exponential time to travel from one cluster to another.  The goal here would be to show in a planted model that Monte Carlo takes exponential time to find the cluster corresponding to the planted solution.  

Finally, both our upper and lower bounds can be improved.  Our upper bound requires that w.h.p.\ all good partitions are correlated with the planted one.  We could obtain better bounds by requiring that this is true w.h.p.\ of \emph{most} good partitions, which would require a lower bound on the typical number of good partitions with large overlap.  Similarly, one can make improvements by focusing on the giant component or the 2-core.  For instance, for $\lambda=1$ (i.e., $\cout = 0$) we have $\dcond = 1$, since for any $\avgdeg > 1$ the graph has w.h.p.\ two giant components, one in each group, while our bounds only give $\dcond \ge 2$.  Progress along these lines was made by~\cite{abbe-sandon-more-groups}, but further improvements seem possible.  

The second moment lower bound could be improved as it was for the $k$-colorability threshold in~\cite{coja-oghlan-vilenchik}.  Indeed, the condensation threshold $\dcond$ for $k$-coloring was determined exactly in~\cite{bapst-condensation-coloring} for sufficiently large $k$.  It is entirely possible that their techniques could work here.  Note that constraint satisfaction problems correspond to zero-temperature models in physics, while the block model with $\cin, \cout \ne 0$ corresponds to a spin system at positive temperature; but some rigorous results have recently been obtained here as well~\cite{bapst-positive-temperature}.

\section*{Acknowledgments.}  This work was supported by the ARO under contract W911NF-12-R-0012 and the John Templeton Foundation.  We are grateful to Emmanuel Abbe, Afonso Bandeira, Amin Coja-Oghlan, Elchanan Mossel, and Joe Neeman for helpful discussions.

\bibliographystyle{alpha}
\bibliography{journals,zp,mark,more}

\end{document}